\documentclass[a4,12pt, leqno]{article}

\usepackage{amsfonts}
\usepackage{amsmath}
\usepackage{theorem}

\newtheorem{theorem}{Theorem}

\newtheorem{corollary}[theorem]{Corollary}

\newtheorem{definition}[theorem]{Definition}

\newtheorem{lemma}[theorem]{Lemma}

\newtheorem{proposition}[theorem]{Proposition}
\newtheorem{remark}[theorem]{Remark}

\newenvironment{proof}[1][Proof]{\noindent\textbf{#1.} }{\ \rule{0.5em}{0.5em}}

\numberwithin{equation}{section}

\begin{document}

\def\E {{\mathbb E}}
\def\P {{\mathbb P}}
\def\R {{\mathbb R}}
\def\Z {{\mathbb Z}}
\def\N {{\mathbb N}}
\def\<{{\langle}}
\def\>{{\rangle}}
\def\eps{\varepsilon}
\def\wh{\widehat}

\renewcommand{\baselinestretch}{1.4}

\title{On The Sub-Mixed Fractional Brownian Motion}
\author{Charles El-Nouty  and Mounir Zili}
\maketitle

\begin{abstract}
Let $\{ S_t^H, \, t \geq 0 \} $ be a linear combination of a Brownian motion
and of an independent sub-fractional Brownian motion with Hurst index $0 < H
< 1$. Its main properties are studied and it is shown that $S^H $ can be
considered as an intermediate process between a sub-fractional Brownian
motion and a mixed fractional Brownian motion. Finally, we determine the
values of $H$ for which $\; S^H$ is not a semi-martingale.
\end{abstract}

\section{Introduction}

Let $\{ B_t^H, t \in \R \}$ be a fractional Brownian motion
(fBm) with Hurst index $0 <~H <~1$, i.e. a centered Gaussian process with
stationary increments satisfying $B_0^H = 0$, with probability 1, and $\E
{(B_t^H)}^2 = \; {\mid t \mid}^{2H}, t \in \R$. We obviously have 
for any real numbers $t$ and $s$
 
\begin{equation}  \label{eq:1}
cov \Bigl( B_t^H, B_s^H \Bigr) = \frac{1}{2} \; \Bigl( \mid t \mid^{2H} +
\mid s \mid^{2H} - \mid t - s \mid^{2H} \Bigr).
\end{equation}

\bigskip 
Consider $\{ B_t, t \in \R \}$ an independent Brownian motion
(Bm) and $(a,b)$ two real numbers such that $(a,b) \neq (0,0)$.

\bigskip
The mixed-fractional Brownian motion (mfBm) is an
extension of a Bm and a fBm. It was introduced in \cite{Cher} in order to solve
some problems in mathematical finance, such as modelling some
arbitrage-free and complete markets. The mfBm $\displaystyle M^H = \{
M_t^H(a,b) ; t \ge 0 \} = \{ M_t^H; t \ge 0 \} $ of parameters $a, b$ and $H$
is defined as follows:
 
\begin{equation*}
\forall t \in {\mathbb{R}}_+, \hspace{5mm} M_t^H = M_t^H(a,b) = a \; B_t + b
\; B_t^H. 
\end{equation*}

\noindent 
We refer also to \cite{CE} and \cite{MZ} for further information
on this process. Let us recall some of its main properties.

\bigskip

\begin{lemma} \label{L:1}
The mfBm $\; (M_t^H(a,b))_{t \in {\mathbb R}_+}$ satisfies the following
properties: 

\begin{itemize}

\item $M^H$ is a centered Gaussian process.
\item $ \forall s \in {\R}_+,  \forall t \in {\R}_+,$
\begin{equation*}
Cov \Big( M_t^H(a,b), M_s^H(a,b) \Big) ) = a^2 (t \wedge s) + \frac{b^2}{2} \Big( t^{2H} + s^{2H} - \mid t-s \mid ^{2H} \Big),
\end{equation*}
where $\displaystyle t \wedge s = \frac{1}{2} \Big( t+s-\mid t-s \mid \Big) .$
\item The increments of the mfBm are stationary.
\end{itemize}
\end{lemma}

\bigskip

In \cite{TB}, the authors suggested a second extension of a Bm, called the
sub-fractional Brownian motion (sfBm), that preserves most of the properties
of the fBm, but not the stationarity of the increments. It is
the stochastic process $\displaystyle \xi^H = \{ \xi_t^H ; t \ge 0 \} $,
defined by:

\begin{equation}
\forall t \in {\mathbb{R}}_+, \hspace{5mm} \xi_t^H = \frac{B_t^H + B_{-t}^H}{\sqrt{2}} ,  \label{eq:4}
\end{equation}

This process arises from occupation time fluctuations of branching particle
systems with Poisson initial condition (see \cite{TB}). Let us state some
results on the sfBm.

\bigskip
\begin{lemma} \label{L:2}
The sfBm $\; (\xi_t^H)_{t \in {\mathbb R}_+}$ satisfies the following
properties:

\begin{itemize}
\item $\xi^H$ is a centered Gaussian process.
\item $ \displaystyle \forall s \in {\R }_+,  \forall t \in {\R }_+,$
\begin{equation*} 
Cov \Big( \xi_t^H, \xi_s^H \Big) ) = s^{2H} + t^{2H} - \frac{1}{2}  \Big( (s+t)^{2H} +  \mid t-s \mid ^{2H} \Big).
\end{equation*}

\item The increments of the smfBm are not stationary.
\end{itemize}
\end{lemma}

\bigskip
We can easily remark that, when $H= 1/2, \; \xi^{1/2}$ is a Bm.

\bigskip 
We refer to \cite{TB, CE1, TU} for further
information on this process.

\bigskip
In the spirit of \cite{TB} and \cite{MZ}, we introduce a new process, that
we will call the sub-mixed fractional Brownian motion (smfBm). More
precisely, the smfBm of parameters $a,b$ and $H$, is a process $\; S^H = \{ S_t^H(a,b) ; t \ge 0 \} = \{ S_t^H; t \ge 0 \} $,
defined by:

\begin{equation}
\forall t \in {\mathbb{R}}_+, \hspace{5mm} S_t^H = S_t^H(a,b) = \frac{a \; ( B_t
+ B_{-t} ) + b \; (B_t^H + B_{-t}^H )}{\sqrt{2}} = a \; \xi_t + b \ \xi_t^H,
\label{eq:6}
\end{equation}
\noindent where $\xi$ is a Bm, obviously independent of $\xi^H$.%
\newline

When $a = 0$ and $b = 1, \; S^H= \xi^H $ is a sfBm. When $a = 1$
and $b = 0, \; S^H = \xi $ is a Bm.

\bigskip 
So the smfBm is clearly an extension of the sfBm and the Bm. This
is the flavor of this process. We will show first that it has the same
properties as the sfBm. Then, we will prove that it has also some of the
main properties of the mfBm, but that its increments are not stationary; they are
more weakly correlated on non-overlapping intervals. Hence $S^H$ may be
considered as being intermediate between the sfBm and the mfBm. This is why
we call it the smfBm.

\bigskip
The aim of this paper is to study on one hand some key properties of
the smfBm and on the other hand its martingale properties. The motivation of
the authors is to measure the consequences of the lack of increments
stationarity. 

\bigskip
In section 2, the main properties of the smfBm are studied, namely:
\begin{itemize}

\item the mixed-self-similarity property (see \cite{MZ}),

\item the non Markovian property,

\item the increments non stationarity property,

\item the correlation coefficient and the influence of the parameters $a$ and $b$ on it, 

\item the comparison between the mfBm and the smfBm covariance
properties.

\end{itemize}

\noindent
 Finally it is shown in section 3 that the smfBm is a semi-martingale if and only if 
\begin{equation*}
b = 0 \; \; \; \mbox{or} \; \; \ H \, \in \, \{1/2\} \, \cup \, ]3/4,1[.
\end{equation*}

\section{Main properties}

\subsection{Basic properties}

The following lemmas describe the basic properties of the smfBm.

\begin{lemma} \label{L:3}
\textit{The smfBm $\displaystyle(S_{t}^{H}(a,b))_{t\in {\mathbb{R}}%
_{+}}$ satisfies the following properties: }

\begin{itemize}
\item \textit{$S^H $ is a centered Gaussian process.}

\item $\forall s\in {\mathbb{R}}_{+},\;\forall t\in {\mathbb{R}}_{+}$, 
\begin{equation}
\begin{array}{rcl}
&  & Cov\Big(S_{t}^{H}(a,b),S_{s}^{H}(a,b)\Big)=a^{2}\; \left( s\wedge
t\right) \\ 
\noalign{\vskip 3mm} & + & b^{2}\; \left( t^{2H}+s^{2H}-\frac{1}{2}\left(
\left( s+t\right) ^{2H}+\mid t-s\mid ^{2H}\right) \right) .%
\end{array}
\label{eq:7}
\end{equation}

\item 
\begin{equation}
\forall t\in {\mathbb{R}}_{+},\hspace{3mm}{\mathbb{E}}\Big(\big(%
S_{t}^{H}(a,b)\big)^{2}\Big)=a^{2}t+b^{2}\; \left( (2-2^{2H-1})\quad
t^{2H}\right) .  \label{eq:8}
\end{equation}
\end{itemize}
\end{lemma}

\begin{proof}
It is a direct consequence of lemma \ref{L:2}.
\end{proof}

\noindent N{\scriptsize OTATION.} Let $(X_t)_{t \in {\mathbb{R}}_+}$ and $%
(Y_t)_{t \in {\mathbb{R}}_+}$ be two processes defined on the same
probability space $(\Omega , F, {\mathbb{P}})$. The notation $\{ X_t \} 
\overset{\Delta }{=} \{ Y_t \} $ will mean that $(X_t)_{t \in {\mathbb{R}}%
_+} $ and $(Y_t)_{t \in {\mathbb{R}}_+}$ have the same law.

\bigskip Let us check the mixed-self-similarity property of the smfBm, which
was introduced in \cite{MZ} in the mfBm case.

\begin{lemma} \label{L:4}
{\it For any $h > 0$, \hspace{2mm}
 $\displaystyle  \{ S_{ht}^H(a,b) \}  \stackrel{\Delta }{=}  \Big\{
 S_t^H \Big( a h^{1/2} ,   b h^{H} \Big) \Big\} .$ }
\end{lemma}

\begin{proof} For fixed $h>0$ , the processes $%
\{S_{ht}^{H}(a,b)\}$, and $\Big\{S_{t}^{H}\Big(ah^{1/2},bh^{H}\Big)\Big\}$
are centered Gaussian. Therefore, one has only to prove that they have the
same covariance function. We have for any $s$ and $t$ in ${\mathbb{R}}_{+}$:

\begin{equation*}
\begin{array}{rcl}
Cov\Big(S_{ht}^{H}(a,b),S_{hs}^{H}(a,b)\Big) & = & (h^{1/2}a)^{2}\;
\left( s\wedge t\right) \\ 
\noalign{\vskip 3mm} & + & (h^{H}b)^{2}\; \left( t^{2H}+s^{2H}-\frac{1}{2}%
\left( \left( s+t\right) ^{2H}+\mid t-s\mid ^{2H}\right) \right) \\ 
\noalign{\vskip 3mm} & = & Cov\Bigg(%
S_{t}^{H}(ah^{1/2},bh^{H}),S_{s}^{H}(ah^{1/2},bh^{H})\Bigg).%
\end{array}%
\end{equation*}%

\bigskip
This ends the proof of the lemma. 

\end{proof}

\begin{lemma} \label{L:5}
 \textit{For any $H\in \Big]0,1\Big[\setminus \Big\{\frac{1}{2}%
\Big\}
$, $a\in {\mathbb{R}}$ and $b\in {\mathbb{R}}^{\ast }$, $(S_{t}^{H}(a,b))_{t%
\in {\mathbb{R}}_{+}}$ is not a Markovian process. }
\end{lemma}

\begin{proof} By lemma \ref{L:3}, $S^{H}$ is a centered
Gaussian process such that $\mathbb{E}\left( S_{t}^{H}\right) ^{2}>0$ for
all $t>0$. Then, if $S^{H}$ were a Markovian process, according to \cite{Re},
for all $0<s<t<u$ we would have:

\begin{equation}
Cov\Big(S_{s}^{H},S_{u}^{H}\Big)Cov\Big(S_{t}^{H},S_{t}^{H}\Big)=Cov\Big(%
S_{s}^{H},S_{t}^{H}\Big)Cov\Big(S_{t}^{H},S_{u}^{H}\Big).  \label{eq:9}
\end{equation}

We get by lemma  \ref{L:3},

\begin{eqnarray*}
Cov\Big(S_{s}^{H},S_{t}^{H}\Big) &=&a^{2}s+b^{2}s^{2H}+b^{2}\left( t^{2H}-%
\frac{1}{2}\left( t+s\right) ^{2H}-\frac{1}{2}\left( t-s\right) ^{2H}\right)
, \\
Cov\Big(S_{t}^{H},S_{t}^{H}\Big) &=&a^{2}t+b^{2}\left( 2-2^{2H-1}\right)
\quad t^{2H}, \\
Cov\Big(S_{t}^{H},S_{u}^{H}\Big) &=&a^{2}t+b^{2}t^{2H}+b^{2}\left( u^{2H}-%
\frac{1}{2}\left( u+t\right) ^{2H}-\frac{1}{2}\left( u-t\right) ^{2H}\right)
, \\
Cov\Big(S_{s}^{H},S_{u}^{H}\Big) &=&a^{2}s+b^{2}s^{2H}+b^{2}\left( u^{2H}-%
\frac{1}{2}\left( u+s\right) ^{2H}-\frac{1}{2}\left( u-s\right) ^{2H}\right)
.
\end{eqnarray*}

Let $s$ be fixed and set $u=e^{t}$. When $ t \rightarrow + \infty,$ Taylor expansions yield

\begin{equation*}
t^{2H}-\frac{1}{2}\left( t+s \right) ^{2H}-\frac{1}{2}\left( t-s\right)
^{2H}=-H\left( 2H-1\right) \frac{s^{2}}{t^{2-2H}}+o\left( \frac{s^{2}}{%
t^{2-2H}}\right) ,
\end{equation*}%
\noindent
and
\begin{equation*}
u^{2H}-\frac{1}{2}\left( u+t\right) ^{2H}-\frac{1}{2}\left( u-t\right)
^{2H}=-H\left( 2H-1\right) \frac{t^{2}}{e^{(2-2H)t}}+o\left( \frac{t^{2}}{%
e^{(2-2H)t}}\right).
\end{equation*}

\bigskip
\noindent
Therefore, for $(h,x) \in \{ (s,t), (t,u), (s,u) \}$,
\begin{equation*}
\underset{x\rightarrow \infty }{\lim }\left( x^{2H}-\frac{1}{2}\left(
x+h\right) ^{2H}-\frac{1}{2}\left( x-h\right) ^{2H}\right) =0.
\end{equation*}

To verify (\ref{eq:9}), a necessary condition is that, when $b\neq 0$,

\begin{equation*}
\underset{t\rightarrow \infty }{\lim }\left( Cov\Big(S_{s}^{H},S_{u}^{H}\Big)%
Cov\Big(S_{t}^{H},S_{t}^{H}\Big)-Cov\Big(S_{s}^{H},S_{t}^{H}\Big)Cov\Big(%
S_{t}^{H},S_{u}^{H}\Big)\right) =0,
\end{equation*}%
that is

\begin{equation*}
\left( a^{2}s+b^{2}s^{2H}\right) \underset{t\rightarrow \infty }{\lim }%
\left( \left( a^{2}t+b^{2}\left( 2-2^{2H-1}\right) t^{2H}\right) -\left(
a^{2}t+b^{2}t^{2H}\right) \right) =0.
\end{equation*}

The last equality is satisfied when 
\begin{equation*}
2-2^{2H-1}=1\Leftrightarrow H=\frac{1}{2}.
\end{equation*}

\bigskip The proof of lemma \ref{L:5} is complete. 
\end{proof}

\bigskip

\begin{proposition} \label{p:6}  \textit{Second moment of increments:}\newline
We have for all $(s,t)\in {\mathbb{R}}_{+}^{2},$ $s\leq t$,

\begin{itemize}
\item 
\begin{equation}
\begin{array}{rcl}
&  & E \Big( S_t^H(a,b) - S_s^H(a,b) \Big) ^2 = a^2 (t -s) \\ 
\noalign{\vskip 2mm} & + & b^2 \Bigg( - 2^{2H-1} (t^{2H} + s^{2H}) +
(t+s)^{2H} + (t -s)^{2H} \Bigg).%
\end{array}
\label{eq11}
\end{equation}

\item 
\begin{equation}
a^2(t-s) + b^2 \gamma (t-s)^{2H} \le E \Big( S_t^H(a,b) - S_s^H(a,b) \Big) %
^2 \le a^2 (t-s) + b^2 \nu (t-s)^{2H},  \label{eq12}
\end{equation}
where 
\begin{equation*}
\gamma = \left\{ 
\begin{array}{rcl}
\displaystyle 2- 2^{2H-1} & if & \displaystyle H > \frac{1}{2}, \\ 
\noalign{\vskip 2mm} \displaystyle 1 & if & \displaystyle H \le \frac{1}{2},
\\ 
&  & 
\end{array}
\right.  
\end{equation*}
and 
\begin{equation*}
\nu = \left\{ 
\begin{array}{rcl}
\displaystyle 1 & if & \displaystyle H \ge \frac{1}{2}, \\ 
\noalign{\vskip 2mm} \displaystyle 2- 2^{2H-1} & if & \displaystyle H < 
\frac{1}{2} . \\ 
&  & 
\end{array}
\right. 
\end{equation*}
\end{itemize}
\end{proposition}

\begin{proof} Equality (\ref{eq11}) is a direct
consequence of equalities (\ref{eq:7}) and (\ref{eq:8}). So let us check the
inequalities (\ref{eq12}). Setting 
\begin{equation}
A(s,t)=\Bigg(\frac{t+s}{2}\Bigg)^{2H}-\frac{t^{2H}+s^{2H}}{2},  \label{eq15}
\end{equation}%
we can write 
\begin{equation}
E\Big(S_{t}^{H}(a,b)-S_{s}^{H}(a,b)\Big)^{2}-a^{2}(t-s)=b^{2}\Bigg(%
(t-s)^{2H}+2^{2H}A(s,t)\Bigg).  \label{eq16}
\end{equation}

We get by convexity that, if $H\leq \frac{1}{2}$, then $A(s,t)\geq 0$ and
consequently 
\begin{equation}
a^{2}(t-s)+b^{2}(t-s)^{2H}\leq E\Big(S_{t}^{H}(a,b)-S_{s}^{H}(a,b)\Big)^{2},
\label{eq17}
\end{equation}%
and if $\displaystyle H\geq \frac{1}{2}$, then $A(s,t)\leq 0$ and
consequently 
\begin{equation}
a^{2}(t-s)+b^{2}(t-s)^{2H}\geq E\Big(S_{t}^{H}(a,b)-S_{s}^{H}(a,b)\Big)^{2}.
\label{eq18}
\end{equation}

To complete the proof of proposition \ref{p:6}, we need a technical
lemma.

\bigskip
\begin{lemma} \label{L:7} Consider, for any $s>0$, the function $f$ defined as follows 
\begin{equation*}
f(x)=-2^{2H-1}((x+s)^{2H}+s^{2H})+(x+2s)^{2H}-(1-2^{2H-1})\ x^{2H},\quad
x\geq 0.
\end{equation*}
If $H<\frac{1}{2}$, $f$ is a negative decreasing function, whereas, if $H>%
\frac{1}{2}$, $f$ is a positive increasing one.
\end{lemma}

\begin{proof} (of lemma \ref{L:7} ) It is clear that 
$\ f(0)=0$. We get for $x>0$

\begin{equation*}
f^{\prime }\left( x\right) =H\ x^{2H-1}g(x),
\end{equation*}

\noindent where

\begin{equation*}
g(x) = -2^{2H} \Big( \frac{s}{x} + 1 \Big) ^{2H-1} + 2 \Big( \frac{2s}{x} +
1 \Big) ^{2H-1} - (2-2^{2H}).
\end{equation*}

We have 
\begin{equation*}
g^{\prime }(x) = \frac{(2H-1)s}{x^2} \Bigg( 2^{2H} \Big( \frac{s}{x} + 1 %
\Big) ^{2H-2} -4 \Big( \frac{2s}{x} + 1 \Big) ^{2H-2} \Bigg) .
\end{equation*}

Let us consider the two following cases:\newline

\textit{Case $1$:} $H<\frac{1}{2}$. Since $2H-1<0$, $2-2^{2H}>0$ and
consequently 
\begin{equation*}
\lim_{x\rightarrow 0^{+}}g(x)=-(2-2^{2H})<0
\; \; \; \mbox{and} \; \; \; \lim_{x\rightarrow +\infty}g(x)= 0.
\end{equation*}%
Set 
\begin{equation*}
\ell (x)=\frac{s+x}{2s+x}=\frac{\frac{s}{x}+1}{\frac{2s}{x}+1}.
\end{equation*}%
Since $\ell $ increases from $\frac{1}{2}$ to $1$, $\ell ^{2H-2}$ decreases
from $2^{2-2H}$ to $1$. Then $ \; 
\ell (x)^{2H-2}\leq~2^{2-2H} $, which is equivalent to 
$$
 2^{2H}\Big(\frac{s}{x}+1\Big)%
^{2H-2}-4\Big(\frac{2s}{x}+1\Big)^{2H-2}\leq 0,
$$
\noindent
and consequently $g^{\prime }(x)\geq 0$. Since $g$ increases from $%
-(2-2^{2H})$ to $0$, $g(x)\leq 0$ and therefore $f^{\prime }(x)\leq 0$.
Hence $f$ decreases and $f(x)\leq 0$.\newline

\textit{Case $2$:} $H>\frac{1}{2}$. Following the same lines as in case $1$,
we get $g^{\prime }(x)\leq 0$. Since the function $g$ decreases from $%
-(2-2^{2H})$ to $0$, $f$ increases and $f(x)\geq 0$. This completes the
proof of lemma \ref{L:7}. 

\end{proof}

Combining $(\ref{eq17})$ and $(\ref{eq18})$ with (\ref{eq16}) and lemma %
\ref{L:7}, we complete the proof of proposition \ref{p:6}. 

\end{proof}

\begin{remark} As a consequence of proposition \ref{p:6}, we
insist on the fact that the smfBm does not have stationary increments, but
this property is replaced by inequalities (\ref{eq12}).
\end{remark}


\subsection{Study of the correlation coefficient of the smfBm increments}

N{\scriptsize OTATION.} Let $\displaystyle X$ and $\displaystyle Y$ be two
random variables defined on the same probability space $(\Omega , F, {%
\mathbb{P}})$ such that $ V(X) \times V(Y) \neq 0$. We denote the correlation coefficient $\rho (X,Y)$ by: 
\begin{equation*}
\rho (X,Y) = \frac{ Cov (X,Y)}{ \sqrt{V(X)} \sqrt{V(Y)}} .
\end{equation*}

\begin{lemma} \label{L:8}
We have for $a \in {\mathbb R}, b \in {\mathbb R}^* , s \in {\mathbb R}_+, t \in {\mathbb R}_+$ and  
 $ h \in 
{\mathbb R}_+ $ such that $\displaystyle 0 < h \le t-s$,

\begin{equation}
 \rho \Big( S_{t+h}^H - S_t^H, 
S_{s+h}^H - S_s^H \Big)  
= \frac{ \gamma (s,t,h)}{ \sqrt{ \Big( 2\frac{a^2}{b^2} h + \alpha (s,h) \Big) \Big( 2\frac{a^2}{b^2} h + \alpha (t,h) \Big) }},   
\label{eq19}
\end{equation}
\noindent
where
\begin{equation*} 
 \begin{array}{rcl}
\displaystyle \gamma (s,t,h) &=& \displaystyle  \Bigg( (t-s + h )^{2H}   - 2 (t-s )^{2H} 
+ ( t-s -h )^{2H}  \\
\noalign{\vskip 3mm}
& - & \displaystyle    (t+s) ^{2H}   + 2 (t+s+h) ^{2H} 
- ( t+s+2h ) ^{2H} \Bigg) ,
\end{array} 
\end{equation*} 
and $ \; \alpha (s,h) =  -2^{2H} \Bigl( (s+h)^{2H} + s^{2H} \Bigr) + 2 \,(2s+h)^{2H} + 2 h^{2H} .$

\end{lemma}	  

\bigskip
\begin{proof} 
We have by equality (\ref{eq11}) 

\begin{equation} \label{eq100}
\begin{array}{rcl}
\E \Bigl( S_{t+h}^H - S_t^H \Bigr)^2 &=& a^2 \, h + b^2 \; \Biggl( - 2^{2H-1} \; \Bigl( (t+h)^{2H} + t^{2H} \Bigr) \\
                                     &+& (2t+h)^{2H} + h^{2H} \Biggr)\\
                                     &=& a^2 \, h + \frac{b^2}{2} \; \alpha(t,h). 
\end{array}
\end{equation}

\bigskip
Recall that a Bm has independent increments and that the processes $\xi^H$ and $\xi$ are independent. Then, we have

\bigskip
\noindent
$ Cov \Big( S_{t+h}^H - S_t^H, 
S_{s+h}^H - S_s^H \Big)  
= 
b^2 \;
Cov \Big( \xi_{t+h}^H - \xi_t^H, 
\xi_{s+h}^H - \xi_s^H \Big)$,

\bigskip
\noindent
and we get by using lemma \ref{L:2}
 
\begin{equation} 
\label{eq101}
Cov \Big( S_{t+h}^H - S_t^H, S_{s+h}^H - S_s^H \Big) = \frac{b^2}{2} \; \gamma(s,t,h).
\end{equation}

\bigskip
Combining (\ref{eq100}) with (\ref{eq101}), we complete the proof of lemma \ref{L:8}.

\end{proof}

\bigskip
\begin{corollary}  \label{C:9} Let 
$ a \in {\mathbb R}$  and 
$ b \in {\mathbb R}^*$. Then, the increments of 
$\displaystyle (S_t^H(a,b))_{t \in {\mathbb R}_+}$  are positively
correlated for $\displaystyle \frac{1}{2} < H < 1$, uncorrelated for 
$\displaystyle H = \frac{1}{2} $, and negatively correlated for
$\displaystyle 0 < H < \frac{1}{2} $.
\end{corollary}

\bigskip

\begin{proof} Let us write the function $\gamma $
given in \eqref{eq19} as 
\begin{equation*}
\displaystyle \gamma (s,t,h) = f(t-s) - f(t+s+h),
\end{equation*}
where $f: x \longmapsto (x+h)^{2H} - 2 x^{2H} + (x-h)^{2H}$.
We have for every $x > 0$ 
\begin{equation*}
f^{\prime }(x) = 2H \Big( (x+h)^{2H-1} - 2 x^{2H-1} + (x-h)^{2H-1} \Big) .
\end{equation*}

The study of the convexity of the function $\displaystyle x \longmapsto x^{2H-1}$ 
enables us to determine the sign of $f^{\prime }$ and therefore the monotony of $f$. 
This ends the proof of corollary~\ref{C:9}.

\end{proof}

\bigskip
As a direct consequence of lemma \ref{L:8}, we get the following corollary.

\bigskip
\begin{corollary} \label{C:10}
Assume that $ b \neq 0$. Then, 
$ \mid   \rho \Big( S_{t+h}^H - S_t^H, 
S_{s+h}^H - S_s^H \Big)  
    \mid$ is a decreasing function of $ \frac{a^2}{b^2} $.
\end{corollary}

\bigskip
Thus, to model some phenomena, we can choose the parameters $ H, a $ and $b$ in such a manner that 
$ \{ S_t^H(a,b), \, t \geq 0 \}$ yields a good model, taking the sign and the level of correlation of the 
phenomenon of interest into account. For example, let us assume that the parameters $ H$ and $a$ are known with $ H > 1/2$, and 
$ b \neq 0$ is not known. Combining corollary \ref{C:9} with corollary \ref{C:10}, we obtain that
the correlation of the increments of $S_H$ increases with $\mid b \mid$.


\subsection{Some comparisons between mfBm and smfBm}

Set for any $\displaystyle s, t > 0$ 
\begin{equation*}
R_H(s,t) = Cov \Big( M_t^H(a,b), M_s^H(a,b) \Big) \hspace{2mm} \mathrm{and } 
\hspace{2mm} C_H(s,t) = Cov \Big( S_t^H(a,b), S_s^H(a,b) \Big).
\end{equation*}

Let us compare $R_H$ and $C_H$.

\bigskip
\begin{lemma} \label{L:12} 
\begin{itemize}
\item $\displaystyle C_H(s,t) \geq 0.$ 
\item If $\displaystyle H > \frac{1}{2}, $ $\displaystyle C_H(s,t) < R_H(s,t)$.
\item If $\displaystyle H = \frac{1}{2}, $ $\displaystyle C_{1/2}(s,t) = R_{1/2}(s,t)$.
\item If $\displaystyle H < \frac{1}{2}, $ $\displaystyle C_H(s,t) > R_H(s,t) $.

\end{itemize}

\end{lemma}

\bigskip
\begin{proof}

Let us show the first assertion. We have by equality (\ref{eq11})
\begin{equation*}
\frac{1}{2} \Bigg( - 2^{2H} (t^{2H} + s^{2H}) + 2 \,(t+s)^{2H} + 2 \mid t -s
\mid ^{2H} \Bigg) = E \Big( S_t^H(0,1) - S_s^H(0,1) \Big) ^2 \geq 0.
\end{equation*}

Thus, we get for every $ 0 <s^{'} < t^{'}$ 
\begin{equation*}
2 \, (t^{'}+s^{'})^{2H} + 2 \, (t^{'} -s^{'}) ^{2H} \geq 2^{2H} (t^{'2H} + s^{'2H}) .
\end{equation*}
By applying this inequality with $t^{'}= t+s$ and $s^{'}= t-s$, we
obtain 
\begin{equation*}
2 \, (t^{2H} + s^{2H}) \geq (t+s)^{2H} + (t-s)^{2H}.
\end{equation*}

This implies by equality (\ref{eq:7}) that $\displaystyle C_H(s,t) \geq 0.$

 \bigskip
 For the next three assertions, we observe that, by using the expressions of $C_H$ and $%
R_H $,
\begin{equation*}
C_H(s,t) - R_H(s,t) = \frac{b^2}{2} \; \Big( t^{2H} + s^{2H} - (s+t)
^{2H} \Big).
\end{equation*}

When $H=\frac{1}{2}, C_{1/2} = R_{1/2}$.
When $H \neq \frac{1}{2}$, set $ u = \frac{s}{t}, \; 0 \leq u \leq 1$. We get 
\begin{equation*}
C_H(s,t) - R_H(s,t) =  \frac{b^2}{2} \; t^{2H} \; g(u), 
\end{equation*}
\noindent
where $ g(u) = 1 + u^{2H} - (1+u)^{2H}$.

\bigskip
The study of the function $g$ completes the proof of the lemma.

\end{proof}
 
Let us turn to the expressions of the covariances of
the mfBm and the smfBm increments on non-overlapping intervals. To this aim, 
denote for $0 \le u < v \le s < t,$ 
\begin{equation*}
R_{u,v,s,t}= Cov \Big( M_v^H(a,b)-M_u^H(a,b), M_t^H(a,b)-M_s^H(a,b) \Big)
\end{equation*}
and 
\begin{equation*}
C_{u,v,s,t}= Cov \Big( S_v^H(a,b)-S_u^H(a,b), S_t^H(a,b)-S_s^H(a,b) \Big).
\end{equation*}

We deduce easily from lemma \ref{L:1} and lemma \ref{L:3} the following result.

\bigskip
\begin{lemma} \label{L:13} We have 
 
\begin{equation}
 R_{u,v,s,t} = \frac{b^2}{2} \; \Big( (t-u)^{2H} + (s-v)^{2H} 
   - (t-v)^{2H}  - (s-u)^{2H} \Big),
   \label{eq:23}
   \end{equation}
   
 \begin{equation}
\begin{array}{rcl}
\displaystyle C_{u,v,s,t} &=& \displaystyle  \frac{b^2}{2} \; \Big( (t+u)^{2H} + (t-u)^{2H} + (s+v)^{2H} + (s-v)^{2H} \\
\noalign{\vskip 2mm}
&-& \displaystyle  (t+v)^{2H} - (t-v)^{2H} - (s+u)^{2H} - (s-u)^{2H} \Big) .
\end{array} 
\label{eq:24}
\end{equation}

\end{lemma}

\bigskip
Let us show that the covariances of the mfBm and
the smfBm increments on non-overlapping intervals have the same sign but,
those of the smfBm are smaller in absolute value than those of the mfBm.

\bigskip
\begin{corollary} \label{C:14}
We have for $0 \le u < v \le s < t,$ that $R_{u,v,s,t}$ and $C_{u,v,s,t}$ are strictly positive (respectively 
strictly negative) for $ H > 1/2 $ (respectively $ H < 1/2$). Moreover, 
$ C_{u,v,s,t} < R_{u,v,s,t} $ (respectively $>$).
 
\end{corollary}

\bigskip
\begin{proof} First, we have $ 0 \le u < v \le s < t$ 
$$
R_{u,v,s,t} = \frac{b^2}{2} \; \Bigl( g_1(v) - g_1(u) \Bigr),
$$
\noindent
where $ g_1(x) = (s-x)^{2H} - (t-x)^{2H}, \; u \le x \le v$.

We have
$$
g_1^{'}(x) = 2H \; \Bigl( - (s-x)^{2H-1} + (t-x)^{2H-1} \Bigr).
$$

When $ H < 1/2, g_1^{'} \le 0$. Then $ g_1$ decreases and therefore $ R_{u,v,s,t} \le 0$.
When $ H > 1/2, g_1^{'} \ge 0$. Then $ g_1$ increases and therefore $ R_{u,v,s,t} \ge 0$.

\bigskip
Next we have for $ 0 \le \ u < v \le s < t$ 
$$
C_{u,v,s,t} = \frac{b^2}{2} \Bigl( g_2(t) - g_2(s) \Bigr)
$$
\noindent
where $
g_2(x) = -(x+v)^{2H} - (x-v)^{2H} + (x+u)^{2H} + (x-u)^{2H}, \; s \le x \le t$.

We have
$$
g_2^{'}(x) = 2H \; \Bigl( g_3(u) - g_3(v) \Bigr),
$$
\noindent
where $ g_3(y) = (x+y)^{2H-1} + (x-y)^{2H-1}, \; u \le y \le v$.

We have
$$
g_3{'}(y) = (2H-1) \; 
\Bigl( (x+y)^{2H-2} - (x-y)^{2H-2} \Bigr).
$$

When $ H < 1/2, g_3^{'} > 0$. Since $ g_3$ increases, $g_2^{'} < 0$ and therefore $ g_2$ decreases. Thus 
$ C_{u,v,s,t} \le 0$. When $ H > 1/2, g_3^{'} < 0$. Since $g_3$ decreases, $ g_2^{'} > 0$ and therefore 
$ g_2$ increases. Thus $ C_{u,v,s,t} \ge 0$.

\bigskip
Finally let us denote by 
$ D_(u,v,s,t) $ the quantity defined as follows
\begin{equation}  \label{eq500}
\begin{array}{crl}
D_{u,v,s,t} & = & C_{u,v,s,t} - R_{u,v,s,t} \\
& = & \frac{b^2}{2} \; \Bigl( (t+u)^{2H}
- (t+v)^{2H} + (s+v)^{2H} - (s+u)^{2H} \Bigr) \\
& = & \frac{b^2}{2} \; \Bigl( g_4(t) -g_4(s) \Bigr),
\end{array}
\end{equation}
\noindent
where $g_4(x) = (x+u)^{2H} - (x+v)^{2H}, \; s \le x \le t.$

Let us remark that, when $ H > 1/2, \; g_4 $ decreases, and when  $ H < 1/2, \; 
g_4 $ increases. This ends the proof of the lemma.

\end{proof}

\bigskip
\begin{corollary} \label{C:15} We have

\begin{itemize}
\item $\displaystyle \lim_{s,t \rightarrow + \infty } R_{u,v,s,t} = 0 $ if and only if $ 0 < H \le \frac{1}{2}$.
\item For every $\displaystyle  0 < H < 1$, $\displaystyle \lim_{s,t \rightarrow + \infty } C_{u,v,s,t} = 0 $.

\end{itemize}
\end{corollary}

\begin{proof} Combining (\ref{eq:23}) with Taylor expansions, we have as
$ s, t \rightarrow + \infty$,
$$
R_{u,v,s,t} = b^2 \; H \; (v - u) \; 
\Bigl(
\frac{1}{t^{1-2H}} - \frac{1}{s^{1-2H}} \Bigr)
+ o \Bigl(\frac{1}{t^{1-2H}} \Bigr) + o \Bigl( \frac{1}{s^{1-2H}} \Bigr),
$$
\noindent
which proves the first assertion of the corollary.

\bigskip
Let us turn to $ C_{u,v,s,t} $. Combining (\ref{eq:24})
with Taylor expansions, we have as
$ s, t \rightarrow~+ \infty$,
$$
C_{u,v,s,t} = b^2 \; H \; (2H-1) \; (v^2 - u^2) \; 
\Bigl(
\frac{1}{s^{2-2H}} - \frac{1}{t^{2-2H}} \Bigr)
+ o \Bigl(\frac{1}{s^{2-2H}} \Bigr) + o \Bigl( \frac{1}{t^{2-2H}} \Bigr),
$$
\noindent
which completes the proof of corollary \ref{C:15}.   

\end{proof}

\bigskip
In the next lemma, we will show that the increments of the smfBm on
intervals $[u,u+r]$ and $[u+r,u+2r]$ are more weakly correlated than those of the
mfBm.

\bigskip
\begin{lemma} \label{L:16} Assume $ H \neq 1/2$. We have for $u \ge 0$ and $ r > 0$, 
\begin{equation}
\Big| \rho \Bigl( S_{u+r}^H - S_u^H , 
S_{u+2r}^H - S_{u+r}^H \Bigr) \Big| \le \Big| \rho \Big( M_{u+r}^H - M_u^H ,
M_{u+2r}^H - M_{u+r}^H  \Big) \Big| .
\label{eq:25}
\end{equation}
\end{lemma}

\begin{proof}
Combining the definition of $R_{u,v,s,t}$ with (\ref{eq:23}), we get
\begin{equation}
\begin{array}{crl}
\rho \Big( M_{u+r}^H - M_u^H, M_{u+2r}^H - M_{u+r}^H \Big) & = & \frac{
R_{u,u+r,u+r,u+2r} }{\sqrt{ V(M_{u+r}^H - M_u^H ) \; V(M_{u+2r}^H - M_{u+r}^H )}} \\
 &  & \\
& = & \frac{b^2 (2^{2H-1}-1)r^{2H}}{\sqrt{ V(M_{u+r}^H - M_u^H ) \; V(M_{u+2r}^H - M_{u+r}^H \Bigr)}}. 
\label{eq201}
\end{array}
 \end{equation}

\bigskip
Moreover, we get by lemma \ref{L:1}
\begin{equation}
V(M_{u+r}^H - M_u^H ) = V(M_{u+2r}^H - M_{u+r}^H ) = V(M_r^H) = 
a^2 \; r + b^2 \; r^{2H}. 
 \label{eq202}
\end{equation}

\bigskip
\noindent
Then, combining (\ref{eq201}) with (\ref{eq202}), we have 
\begin{equation}
\label{eq203}
\rho \Big( M_{u+r}^H - M_u^H, M_{u+2r}^H - M_{u+r}^H \Big) =
\frac{b^2 (2^{2H-1}-1)r^{2H}}{
a^2 \; r + b^2 \; r^{2H}}.  
 \end{equation}

\bigskip
Let us turn to $ \rho \Bigl( S_{u+r}^H - S_u^H , 
S_{u+2r}^H - S_{u+r}^H \Bigr)$. We have
\begin{equation}
\rho \Big( S_{u+2r}^H - S_{u+r}^H, S_{u+r}^H - S_u^H \Big) = \frac{
C_{u,u+r,u+r,u+2r} }{\sqrt{V(S_{u+2r}^H - S_{u+r}^H)V(S_{u+r}^H - S_{u}^H)}}. 
 \label{eq204}
\end{equation}

\bigskip
Let us consider the two following cases.

\bigskip
\noindent
Case 1. $H < \frac{1}{2}$

\bigskip
By using (\ref{eq203}) and (\ref{eq204}), we can rewrite inequality (\ref{eq:25}) as follows: 
\begin{equation}
\Big| \frac{C_{u,u+r,u+r,u+2r}}{R_{u,u+r,u+r,u+2r} } \Big| 
\le
\frac{ \sqrt{V(S_{u+2r}^H - S_{u+r}^H)V(S_{u+r}^H - S_{u}^H)}}{a^2 \; r + b^2 \; r^{2H}}.   
 \label{eq205}
\end{equation}

\bigskip
Note that by corollary \ref{C:14} and equality (\ref{eq500})
$$
\Big| \frac{C_{u,u+r,u+r,u+2r}}{R_{u,u+r,u+r,u+2r} } \Big| 
=
\frac{C_{u,u+r,u+r,u+2r}}{R_{u,u+r,u+r,u+2r} } 
=
1 + \frac{D_{u,u+r,u+r,u+2r}}{R_{u,u+r,u+r,u+2r}}.
$$

\bigskip
\noindent
Then, (\ref{eq205}) can be rewritten as follows
\begin{equation} 
0 \le 
1 + \frac{D_{u,u+r,u+r,u+2r}}{R_{u,u+r,u+r,u+2r}}
\le
\frac{ \sqrt{V(S_{u+2r}^H - S_{u+r}^H)V(S_{u+r}^H - S_{u}^H)}}{a^2 \; r + b^2 \; r^{2H}}.   
 \label{eq206}
\end{equation}

\bigskip
The second part of proposition \ref{p:6} implies that
$$
a^2 r + b^2 r^{2H} \le 
\sqrt{V(S_{u+2r}^H - S_{u+r}^H) \; V(S_{u+r}^H - S_{u}^H)}.
$$
 
\noindent
Then, to prove (\ref{eq206}), it suffices to show that
$$ 
1 + \frac{D_{u,u+r,u+r,u+2r}}{R_{u,u+r,u+r,u+2r}}
\le
1.
$$

\bigskip
By corollary \ref{C:14}, $ R_{u,u+r,u+r,u+2r} < 0$ and $ D_{u,u+r,u+r,u+2r} >0$. The proof of case 1 is complete.

\bigskip
\noindent
Case 2. $H > \frac{1}{2}$

\bigskip
Combining (\ref{eq203}) with (\ref{eq204}), we get 
$$
\frac{ \rho \Big( S_{u+r}^H - S_u^H, S_{u+2r}^H -
S_{u+r}^H \Big) }{ \rho \Big( M_{u+r}^H - M_u^H, M_{u+2r}^H - M_{u+r}^H \Big) }
=
\frac{
C_{u,u+r,u+r,u+2r} }{\sqrt{V(S_{u+2r}^H - S_{u+r}^H)V(S_{u+r}^H - S_{u}^H)}}
\; \; 
\frac{a^2 \; r + b^2 \; r^{2H}} {b^2 (2^{2H-1}-1)r^{2H}}.
$$

Recall that we have precise expressions of $V(S_{u+r}^H - S_{u}^H) $ and of $V(S_{u+2r}^H - S_{u+r}^H) $
 by the first part of proposition \ref{p:6}
and of $C_{u,v,s,t} $ by equality (\ref{eq:24}).

\bigskip
Set $
 x = 
\frac{2u}{r} $ and denote by $A, B$ and $C$ the functions defined as follows :
\begin{equation*}
A(x) = 2(x+2)^{2H} + (2^{2H} -2) -(x+3)^{2H} -(x+1)^{2H} ,
\end{equation*}
\begin{equation*}
B(x)= 2 - x^{2H} -(x+2)^{2H} + 2 (x+1)^{2H}
\end{equation*}
and \hspace{3mm} $C(x) = 2 - (x+2)^{2H} -(x+4)^{2H} + 2
(x+3)^{2H}.$

\bigskip
Easy computations yield 
$$
C_{u,u+r,u+r,u+2r} = \frac{b^2}{2} \; r^{2H} \; A(x)
$$
$$
 V(S_{u+r}^H - S_{u}^H) = \frac{1}{2} \; \Bigl( 2 a^2 r + b^2 r^{2H} B(x) \Bigr)
 $$
$$
 V(S_{u+2r}^H - S_{u+r}^H) = \frac{1}{2} \; \Bigl( 2 a^2 r + b^2 r^{2H} C(x) \Bigr)
 $$

\bigskip
Then, we have
$$
\frac{ \rho \Big( S_{u+r}^H - S_u^H, S_{u+2r}^H -
S_{u+r}^H \Big) }{ \rho \Big( M_{u+r}^H - M_u^H, M_{u+2r}^H - M_{u+r}^H \Big) }
=
\frac{A(x) \; (a^2 \; r + b^2 \; r^{2H})}{  (2^{2H-1}-1) \; \sqrt{(2 a^2 r 
+ b^2 r^{2H} B(x))(2 a^2 r 
+ b^2 r^{2H} C(x)) }}
$$
$$
=
\frac{A(x)}{  (2^{2H-1}-1) \; \sqrt{ B(x) \, C(x) }}
\;
\frac{a^2 \; r + b^2 \; r^{2H}}{\sqrt{(2 a^2 r / B(x) 
+ b^2 r^{2H})(2 a^2 r / C(x)
+ b^2 r^{2H}) }}.
$$

Since it has been proved in  \cite[see][p. 412]{TB}, that
$$
\frac{A(x)}{  (2^{2H-1}-1) \; \sqrt{ B(x) \, C(x) }}
\le 1,
$$
\noindent
we get
$$
\frac{ \rho \Big( S_{u+r}^H - S_u^H, S_{u+2r}^H -
S_{u+r}^H \Big) }{ \rho \Big( M_{u+r}^H - M_u^H, M_{u+2r}^H - M_{u+r}^H \Big) }
\le
\frac{a^2 \; r + b^2 \; r^{2H}}{\sqrt{(2 a^2 r / B(x) 
+ b^2 r^{2H})(2 a^2 r / C(x)
+ b^2 r^{2H}) }}.
$$

Therefore it suffices to show 
$$
a^2 \; r + b^2 \; r^{2H} \le 2 a^2 r / B(x)
+ b^2 r^{2H} 
\; \; \; \mbox{and} \; \; \;
a^2 \; r + b^2 \; r^{2H} \le 2 a^2 r / C(x)
+ b^2 r^{2H},
$$
\noindent
that is 
$$
0 < B(x) \le 2
\; \; \; \mbox{and} \; \; \;
0 < C(x) \le 2.
$$

Let us show the first double inequality. Since by lemma \ref{L:2}
$$
b^2 \; r^{2H} \; B(x) = 2 \; V\Bigl( b ( \xi^H(u+r) - \xi^H(u)) \Bigr),
$$
\noindent
$ B(x) > 0$. Moreover, since the function $ x \rightarrow x^{2H}$ is convex for $ H > 1/2, \; B(x) \leq 2$. Similarly, we can establish 
$ 0 < C(x) \le 2$.

\bigskip
The proof of the lemma is complete.

\end{proof}

\bigskip
\bigskip
In \cite{MZ}, it was proved that the increments of the mfBm $(M_t^H(a,b))$
are short-range dependent if, and only if $H < \frac{1}{2}$. 
To end this subsection, let us show that for every $\displaystyle H \in ]0,
1[ $, the increments of $\displaystyle (S^H_t(a,b))_{t \in {\R _+}}$ are
short-range dependent. For convenience, let us introduce the following notation 
\begin{equation*}
C(p,n) = C_{p,p+1,p+n,p+n+1},
\end{equation*}
\noindent 
where $p$ and $n$ are integers with $ n \ge 1$.

\bigskip
We get by (\ref{eq:24})
$$
C(p, n) = \frac{b^{2}}{2} \Bigg(  (n+1)^{2H} - 2 n^{2H} + (n-1)^{2H} 
- (2p+n+2)^{2H} + 2 (2p+n+1)^{2H} - (2p+n)^{2H}\Bigg).
$$

\bigskip
A third-order Taylor expansion enables us to state the following lemma.

\bigskip
\begin{lemma}  \label{L:18} 
For any $ 0 < H < 1$ and $ p \in \N$, we have when $ n \rightarrow + \infty$
$$
C(p,n) = \Big( 2(1-H)H(2H-1)(2p+1)b^2 \Big) \; n^{2H-3} + o(n^{2H-3}),
$$
\noindent
and consequently
$$
\sum_{n \ge 1} \mid C(p, n) \mid < + \infty .
$$

\end{lemma}

\bigskip
\bigskip

\section{Semi-martingale properties}

In the sequel, we assume $ b \neq 0$. 
For any process $X$, set 
$$
\Delta_j^n X(t) = X(j t /n) - X ((j-1) t / n), \;
j \in \{1,..n\}.
$$
\noindent  Denote by $A_n$ the quantity defined as follows : 
\begin{equation*}
A_n = \E  \Biggl( \sum_{j=1}^n \; \Biggl( \Delta_j^n S(t) \Biggr)^2 \Biggr) =\sum_{j=1}^n \; %
\E \Biggl( \Delta_j^n S(t) \Biggr)^2.
\end{equation*}

\begin{lemma}\label{L:22}

\begin{itemize}
\item 
If $H < \frac{1}{2}$, then $ \displaystyle \lim_{n \rightarrow + \infty} \; A_n = + \infty$.
\item 
If $H = \frac{1}{2}$, then $ \; A_n = (a^2 + b^2)  \; t$.
\item 
If $H > \frac{1}{2}$, then $ \displaystyle \lim_{n \rightarrow + \infty} \; A_n = a^2 \;  t$.
\end{itemize}
\end{lemma}

\bigskip
\begin{proof} Since the processes $B$ and $B_H$ are independent, we have 
\begin{equation*}
\E \Biggl( \Delta_j^n S(t) \Biggr)^2 = \frac{a^2}{2} \; \E \Biggl( \Delta_j^n B(t) +
\Delta_j^n B(-t) \Biggr)^2 + \frac{b^2}{2} \; \E \Biggl( \Delta_j^n B_H(t) +
\Delta_j^n B_H(-t) \Biggr)^2.
\end{equation*}

Using equality (\ref{eq:1}), direct computations imply 
\begin{equation*}
\E \Biggl( \Delta_j^n S(t) \Biggr)^2 =  a^2 \; \frac{t}{n} + \; b^2 \; 
\frac{t^{2H}}{n^{2H}} + 2^{2H} \; b^2 \; \frac{t^{2H}}{n^{2H}} \; \Biggl( %
\Biggl( \frac{2j-1}{2} \Biggr)^{2H} - \frac{j^{2H} + (j-1)^{2H}}{2} \Biggr),
\end{equation*}
\noindent and hence 
\begin{equation*}
A_n = a^2 \; t + \; b^2 \; t^{2H} \; n^{1-2H} + 2^{2H} \; b^2 \; 
\frac{t^{2H}}{n^{2H}} \; \sum_{j=1}^n \; \Biggl( \Biggl( \frac{2j-1}{2} %
\Biggr)^{2H} - \frac{j^{2H} + (j-1)^{2H}}{2} \Biggr).
\end{equation*}

\bigskip Let us consider the function $f$ defined as follows : 
\begin{equation*}
f(x) = \Biggl( \frac{2x-1}{2} \Biggr)^{2H} - \frac{x^{2H} + (x-1)^{2H}}{2},
\; \; \; x \geq 0.
\end{equation*}

\bigskip We deduce from convexity properties that, when $H < 1/2, \; f(x) > 0$,
when $H = 1/2, \; f(x) =0$ and when $H > 1/2, \; f(x) < 0$. We have also 
\begin{equation*}
f^{^{\prime }}(x) = 2 \; H \; \Biggl( \Biggl( \frac{2x-1}{2} \Biggr)^{2H-1}
- \frac{x^{2H-1} + (x-1)^{2H-1}}{2} \Biggr), \; \; \; x \geq 0.
\end{equation*}

\bigskip To determine $\displaystyle \lim_{n \rightarrow + \infty} \; A_n$,
we have to consider the following three cases.

\bigskip \noindent \textit{Case 1.} $H < 1/2$

\bigskip Since $f^{^{\prime }} \leq 0$, for every $j \in \{1,..,n\}$,
\begin{equation*}
f(j) \geq f(n) = \Biggl( \frac{2n-1}{2} \Biggr)^{2H} - \frac{n^{2H} + (n-1)^{2H}}{2} >
0.
\end{equation*}
When $n$ is large enough, we get
$$
 a^2 \; t + \; b^2 \; t^{2H} \; n^{1-2H} + 2^{2H} \; b^2 \; t^{2H}
\; \Biggl( - \frac{H (2H-1)}{4n} + o\Bigl( \frac{1}{n} \Bigr) \Biggr) \leq A_n, 
$$

\noindent and therefore, since $ b \neq 0$,
\begin{equation*}
\lim_{n \rightarrow + \infty} \; A_n = + \infty.
\end{equation*}

\bigskip \noindent \textit{Case 2.} $H = 1/2$

\bigskip We obviously have 
\begin{equation*}
A_n = (a^2 \; + \; b^2 ) \; t.
\end{equation*}

\bigskip \noindent \textit{Case 3.} $H > 1/2$

\bigskip Since $f^{^{\prime }} \geq 0, \; f$ increases from $f(1) = \frac{1}{%
2^{2H}} - \frac{1}{2}$ to 
\begin{equation*}
f(n) = \Biggl( \frac{2n-1}{2} \Biggr)^{2H} - \frac{n^{2H} + (n-1)^{2H}}{2} <
0.
\end{equation*}
When $n$ is large enough, we get

\bigskip \noindent $\displaystyle 
a^2 \; t + \; b^2 \; t^{2H} \; n^{1-2H} + 2^{2H} \; b^2 \; t^{2H}
\; n^{1-2H} \; \Biggl( \frac{1}{2^{2H}} - \frac{1}{2} \Biggr) $ 
\begin{equation*}
\leq A_n \leq a^2 \; t +  \; b^2 \; t^{2H} \; n^{1-2H} + 2^{2H+} \;
b^2 \; t^{2H} \; \Biggl( - \frac{H (2H-1)}{4n} + o\Bigl( \frac{1}{n} \Bigr) %
\Biggr)
\end{equation*}
\noindent and therefore 
\begin{equation*}
\lim_{n \rightarrow + \infty} \; A_n = a^2 \; t.
\end{equation*}

\bigskip This completes the proof of the lemma. 

\end{proof}

\bigskip
Let us now recall the Bichteler-Dellacherie theorem \cite[see][section VIII.4]{DELL}.

\bigskip
\begin{theorem} \label{T:3.2} 
Assume that a filtration $\displaystyle \mathcal{F} =({\cal F}_{t})_{0 \le t  \le T}, 0 < T < +\infty, $ satisfies the usual assumptions, 
i.e. it is right-continuous, ${\cal F}_{T}$ is complete and ${\cal F}_{0}$ contains all null sets of ${\cal F}_{T}$.
An a.s. right-continuous,  $\displaystyle {\cal F}$-adapted stochastic process $\{X_t, 0 \le t  \le T \}$ is a  $\displaystyle {\cal F}$-semi-martingale 
if and only if 
\begin{equation*}
I_X \Bigl( \beta(\mathcal{F}) \Bigr)
\end{equation*}
\noindent is bounded in $L^0$, where 
\begin{equation*}
\beta(\mathcal{F}) = \Biggl\{ \sum_{j=0}^{n-1} \; f_j \; \mathbf{1}_{
]t_j, t_{j+1}]}, \; \; n \in \mathbb{N}, \; 0 \leq t_0 \leq .. \leq t_n \leq
T, \;
\end{equation*}
\begin{equation*}
\forall \, j, \, f_j \hspace{1mm} \mathrm{is} \hspace{1mm} \mathcal{F}_{t_j} 
\hspace{1mm} \mathrm{measurable} \hspace{1mm}\mathrm{and } \hspace{1mm} \mid
f_j \mid \leq 1 , \mathrm{with} \hspace{1mm} \mathrm{probability} \hspace{1mm%
} 1 \Biggr\},
\end{equation*}
\noindent and 
\begin{equation*}
I_X(\theta) = \sum_{j=0}^{n-1} \; f_j \;\Bigl( X_{t_{j+1}} - X_{t_j} %
\Bigr), \; \; \theta \in \beta(\mathcal{F}).
\end{equation*}

\end{theorem}

\bigskip
Following the same lines as those of \cite{Cher}, we introduce two definitions.

\bigskip
\begin{definition} \label{T:325}
 A stochastic process $\{X_t, 0 \le t  \le T \}$ is a weak semi-martingale with respect to 
 a filtration $\displaystyle \mathcal{F} =({\cal F}_{t})_{0 \le t \le T}$ if $X$ is 
 $\displaystyle {\cal F}$-adapted and 
 $ \displaystyle 
I_X \Bigl( \beta(\mathcal{F}) \Bigr) $ is bounded in $L^0$.

 \end{definition}

\bigskip
We insist on the fact that if a process $X$ is not a weak semi-martingale with respect to its own filtration, then 
it is not a weak semi-martingale with respect to any other filtration.

\bigskip
\begin{definition} \label{T:335}
 Let  $\{X_t, 0 \le t  \le T \}$  be a stochastic process. We call $X$ 
  a weak semi-martingale if it is a weak semi-martingale with respect to its 
 own filtration $\displaystyle \mathcal{F}^X =({\cal F}_{t}^X)_{0 \le t \le T}$. 
 We call $X$ 
  a semi-martingale if it is a semi-martingale with respect to the smallest filtration 
 that contains $\displaystyle \mathcal{F}^X $ and satisfies the usual assumptions.
 
 \end{definition}

\bigskip
Let us determine now the values of $H$ for which the smfBm is not a semi-martingale.

\bigskip
\begin{corollary} 
 \label{C:23}
If $ \; 0 < H < \frac{1}{2}$, 
then the smfBm $S^H(a,b)$ is not a weak semi-martingale.
\end{corollary}

\bigskip
\begin{proof}
 A direct consequence of lemma \ref{L:22} is that
since $0 < H < \frac{1}{2}$ and $b \neq 0$, the quadratic
variation of the smfBm is infinity. To complete the proof of the corollary, 
it suffices to apply proposition 2.2 of \cite[pp. 918-919]{Cher}.

\end{proof}

\bigskip
The study of the case $ H > 3/4$ is based on a result of \cite[p. 348]{Ba}. We insist on the fact that this method is different from 
the one which was used in \cite{Cher}.

\bigskip
\begin{proposition} \label{p:24} 
For every $T > 0, \;  H \in ]\frac{3}{4}, 1[$, and $\; a \neq 0$, the smfBm
$$
S^H(a,b) = \{ S_t^H(a,b), t \in [0,T] \} 
$$
is a semi-martingale equivalent in law to
$ a \times B_t$, where $\{B_t, t \in [0,T]\}$ is a Bm.
\end{proposition}

\begin{proof}
The smfBm $S^H$ can be rewritten as follows 
\begin{equation*}
\forall \; t \in \R^+  \; \; \;
S_t^H(a,b) = a \Big( \xi_t + \frac{b}{a} \xi_t^H \Big)
\end{equation*}
where $\xi $ and $\xi^H $ have been introduced by equation (\ref{eq:6}). Recall that the processes $\xi$ and $\xi^H$ are independent.

\bigskip
The covariance function of
the Gaussian process $\displaystyle  \frac{b}{a} \; \xi^H $
\begin{equation*}
R(s,t) = \frac{b^2}{a^2} \Bigg( t^{2H} + s^{2H} - \frac{1}{2} \Big( 
(s+t) ^{2H} + \mid t -s \mid ^{2H} \Big) \Bigg) ,
\end{equation*}
is twice continuously differentiable on $\displaystyle %
\lbrack 0,T]^2 \setminus \{ (s,t); t =s \}$.

\bigskip
According to \cite[p. 348]{Ba}, it suffices to verify $\displaystyle \frac{\partial ^2R}{\partial s
\partial t } \in L^2([0,T]^2)$, in order to show 
that the process 
\begin{equation*}
\{ \xi_t + \frac{b}{a} \xi_t^H , t \in [0,T] \}
\end{equation*}
is a semi-martingale equivalent in law to a Bm.

\bigskip
We have for any 
$\displaystyle (s,t) \in [0,T]^2 \setminus \{ (s,t); t =s \}$%
\begin{equation*}
\frac{\partial ^2R(s,t) }{\partial s \partial t } = \frac{b^2}{a^2} \;
H(2H-1) \; \Big( \mid t-s \mid ^{2H-2} - (s+t)^{2H-2} \Big).
\end{equation*}

\bigskip
It is easy to check that if $\displaystyle H > \frac{3}{4}, $ then $%
\displaystyle \frac{\partial ^2R}{\partial s \partial t } \in L^2([0,T]^2)$.
This completes the proof of the proposition.

\end{proof}
 
\bigskip
To study the case $H \in ]1/2,3/4]$, we follow the same lines as those of 
\cite{Cher}. But many technical results have to be proved. Let us first recall the definition of a quasi-martingale.

\begin{definition} \label{T:25}
 A stochastic process $\{ X_t, 0 \le t \le T \}$ is a quasi-martingale if $\displaystyle  X_t \in L^1$ for all $ t \in [0,T],$ and 
 \begin{equation*}
 \sup_{\tau } \sum_{j=0}^{n-1} \left\| \E \Big( X_{t_{j+1}} -X_{t_j} | {\cal F}_{t_j}^X \Big) \right\| _1 < + \;\infty ,
 \end{equation*}
 \noindent
 where $\tau $ is the set of all finite partitions
 \begin{equation*}
 0= t_0 < t_1 < ... < t_n = T \hspace{1mm} of \hspace{1mm} [0,T].
 \end{equation*}
 
 \end{definition}

\bigskip
 In the following key lemma, we will specify the relation between
quasi-martingale and weak semi-martingale in the case of our process $S^H$.

\bigskip
\begin{lemma}
 \label{L:26}
 If $ S^H$ is not a quasi-martingale, then it is not a weak semi-martingale.
 \end{lemma}

\bigskip
\begin{proof}

Let us assume that $S^H$ is a weak 
semi-martingale. Then, by theorem 1 of \cite{ST}, we have 
\begin{equation*}
I_{S^H} \Bigl( \beta(\mathcal{F}^{S^H}) \Bigr),
\end{equation*}
\noindent which was defined in theorem \ref{T:3.2}, is bounded in $L^2$, and therefore in $L^1$.

\bigskip But, for any partition $0 = t_0 < t_1< .. <t_n = T$, 
\begin{equation*}
\sum_{j=0}^{n-1 } \; sgn \Biggl( \E \Biggl( S_{t_{j+1}}^H - S_{t_j}^H \mid 
\mathcal{F}_{t_j}^{S^H} \Biggr)\Biggr) \; \mathbf{1}_{ ]t_j, t_{j+1}]} \in
\beta(\mathcal{F}^{S^H}),
\end{equation*}
\noindent and 
\begin{equation*}
\begin{array}{rcl}
&  & \displaystyle \Bigg| \Bigg| I_{S^H} \Biggl( \sum_{j=0}^{n-1} \; sgn %
\Biggl( \E \Biggl( S_{t_{j+1}}^H - S_{t_j}^H \mid \mathcal{F}_{t_j}^{S^H} %
\Biggr)\Biggr) \; \mathbf{1}_{ ]t_j, t_{j+1}]} \Biggr) \Bigg| \Bigg|_1 \\ 
\noalign{\vskip 3mm} & = & \displaystyle \Bigg| \Bigg| \sum_{j=0}^{n-1} \;
sgn \Biggl( \E \Biggl(S_{t_{j+1}}^H - S_{t_j}^H \mid \mathcal{F}_{t_j}^{S^H} %
\Biggr)\Biggr) \; \Bigl( S_{t_{j+1}}^H - S_{t_j}^H \Bigr) \Bigg| \Bigg|_1 \\ 
\noalign{\vskip 3mm} & \geq & \displaystyle \sum_{j=0}^{n-1} \; \E \Biggl( \;
sgn \Biggl( \E \Biggl(S_{t_{j+1}}^H - S_{t_j}^H \mid \mathcal{F}_{t_j}^{S^H} %
\Biggr)\Biggr) \; \Bigl( S_{t_{j+1}}^H - S_{t_j}^H \Bigr) \Biggr) \\ 
\noalign{\vskip 3mm} & = & \displaystyle \sum_{j=0}^{n-1} \;  \E \Biggl( \E %
\Biggl( \; sgn \Biggl( \E \Biggl(S_{t_{j+1}}^H - S_{t_j}^H \mid \mathcal{F}%
_{t_j}^{S^H} \Biggr)\Biggr) \; \Bigl( S_{t_{j+1}}^H - S_{t_j}^H \Bigr) %
\Bigg| \mathcal{F}_{t_j}^{S^H} \Biggr) \Biggr) \\ 
\noalign{\vskip 3mm} & = & \displaystyle \sum_{j=0}^{n-1} \; \E \Biggl( sgn %
\Biggl( \E \Biggl(S_{t_{j+1}}^H - S_{t_j}^H \mid \mathcal{F}_{t_j}^{S^H} %
\Biggr)\Biggr) \E \Biggl( \; \; \Bigl( S_{t_{j+1}}^H - S_{t_j}^H \Bigr) %
\Bigg| \mathcal{F}_{t_j}^{S^H} \Biggr) \Biggr) \\ 
\noalign{\vskip 3mm} & = & \displaystyle \sum_{j=0}^{n-1} \; \Bigg| \Bigg| \E %
\Biggl( S_{t_{j+1}}^H - S_{t_j}^H \mid \mathcal{F}_{t_j}^{S^H} \Biggr) %
\Bigg| \Bigg|_1.%
\end{array}%
\end{equation*}

\bigskip \noindent Then, $S^H$ is a quasi-martingale. The proof of the
lemma is complete. 

\end{proof}

\bigskip 
The following lemmas deal with the two last cases $1/2 < H < 3/4$
and $H = 3/4$.

\bigskip
\begin{proposition} \label{p:27} 
If $\displaystyle H \in \Big] \frac{1}{2}, \frac{3}{4} \Big[$, then the smfBm $S^H(a,b)$ is not a quasi-martingale.
\end{proposition}

\bigskip
\begin{proof} 
For $n \in \N$ and $j \in \{ 1, 2,...,
n \}$, let us denote 
\begin{equation*}
\Delta_{j}^n S^H = S^H_{\frac{Tj}{n}} - S^H_{\frac{T(j-1)}{n}}.
\end{equation*}
Since conditional expectation is a contraction with respect to the $L^1-$ norm, we
have for all $n \in \N$ and all $j= 1, ..., n-1$, 
\begin{equation*}
\left\| \E \Big( \Delta_{j+1}^n S^H | \Delta_{j}^n S^H \Big) \right\| _1 \le
\left\| \E \Big( \Delta_{j+1}^n S^H | \mathcal{F}_{\frac{Tj}{n}}^{S^H} \Big) %
\right\| _1.
\end{equation*}
Moreover, since $\displaystyle \E \Big( \Delta_{j+1}^n S^H | \Delta_{j}^n S^H %
\Big)$ is a centered Gaussian random variable, 
\begin{equation*}
\left\| \E \Big( \Delta_{j+1}^n S^H | \Delta_{j}^n S^H \Big) \right\| _1 = 
\sqrt{\frac{2}{\pi }} \left\| \E \Big( \Delta_{j+1}^n S^H | \Delta_{j}^n S^H %
\Big) \right\| _2.
\end{equation*}

Consequently, 
\begin{equation}
\label{eq50}
\begin{array}{rcl}
\displaystyle \sum_{j=1}^{n-1} \left\| \E \Big( \Delta_{j+1}^n S^H | \mathcal{%
F}_{\frac{Tj}{n}}^{S^H} \Big) \right\| _1 & \ge & \displaystyle \sqrt{\frac{2%
}{\pi }} \sum_{j=1}^{n-1} \left\| \E \Big( \Delta_{j+1}^n S^H | \Delta_{j}^n
S^H \Big) \right\| _2 \\ 
\noalign{\vskip 3mm} & = & \displaystyle \sqrt{\frac{2}{\pi }}
\sum_{j=1}^{n-1} \left\| \frac{ Cov \Big( \Delta_{j+1}^n S^H , \Delta_{j}^n
S^H \Big) }{ Cov \Big( \Delta_{j}^n S^H , \Delta_{j}^n S^H \Big) }
\Delta_{j}^n S^H \right\| _2 \\ 
\noalign{\vskip 3mm} & = & \displaystyle \sqrt{\frac{2}{\pi }}
\sum_{j=1}^{n-1} \frac{ Cov \Big( \Delta_{j+1}^n S^H , \Delta_{j}^n S^H %
\Big) }{ \sqrt{Cov \Big( \Delta_{j}^n S^H , \Delta_{j}^n S^H \Big) }} %
:=  \displaystyle \sqrt{\frac{2}{\pi }}  \; I_n.
\end{array}%
\end{equation}

\bigskip
We have by lemma \ref{L:13},

\begin{equation*}
\begin{array}{rcl}
\displaystyle Cov \Big( \Delta_{j+1}^n S^H , \Delta_{j}^n S^H \Big) & = & %
\displaystyle C_{\frac{T(j-1)}{n} , \frac{Tj}{n}, \frac{Tj}{n},\frac{T(j+1)}{%
n}} \\ 
\noalign{\vskip 3mm} & = & \displaystyle \frac{b^2 \; T^{2H}}{2 \; n^{2H}} \; \Big( %
2^{2H} (2j^{2H} + 1) - 2 - (2j +1)^{2H} -(2j-1)^{2H} \Big) .%
\end{array}%
\end{equation*}

\bigskip
Combining proposition \ref{p:6} with the fact that $2H > 1$, we get

\begin{equation*}
\begin{array}{rcl}
&  & \displaystyle Cov \Big( \Delta_{j}^n S^H , \Delta_{j}^n S^H \Big) \\ 
\noalign{\vskip 2mm} & = & \displaystyle a^2 \frac{T}{n} + \frac{b^2 \; T^{2H}}{%
n^{2H}} \Bigg( - 2^{2H-1} (j^{2H} + (j-1)^{2H}) + (2j-1)^{2H} + 1 \Bigg) \\ 
\noalign{\vskip 2mm} & \le & \displaystyle \displaystyle \frac{1}{n} \Bigg( %
a^2 \; T + b^2 \; T^{2H} \Big( - 2^{2H-1} (j^{2H} + (j-1)^{2H}) + (2j-1)^{2H} + 1 %
\Big) \Bigg). \\ 
&  & 
\end{array}%
\end{equation*}

\bigskip
\noindent
Then, 
\begin{equation*}
I_n \ge \frac{b^2 \;  T^{2H}}{2\; n^{2H-\frac{1}{2}}} \sum_{j=1}^{n-1} \frac{u_j}{%
v_j} 
=
\frac{b^2 \;  T^{2H}}{2\; n^{2H-\frac{3}{2}}} \times  \Biggl(  \frac{1}{n} \sum_{j=1}^{n-1} \frac{u_j}{v_j} \Biggr)
,
\end{equation*}
where we have for any $n \in \N^*$,

\begin{equation*}
u_n= 2^{2H} (2n^{2H} + 1) - 2 - (2n +1)^{2H} -(2n-1)^{2H}
\end{equation*}
and 
\begin{equation*}
v_n = \sqrt{ a^2 \; T + b^2 \; T^{2H} \Big( - 2^{2H-1} (n^{2H} + (n-1)^{2H}) +
(2n-1)^{2H} + 1 \Big) } .
\end{equation*}

 \bigskip
Since
\begin{equation*}
\lim_{n \rightarrow + \infty } \frac{u_n}{v_n} = \frac{2^{2H}-2}{\sqrt{a^2\; T
+ b^2 \; T^{2H}}} ,
\end{equation*}
\noindent 
we have by C\'{e}saro
theorem that
\begin{equation*}
\lim_{n \rightarrow + \infty } \frac{1}{n} \sum_{j=1}^{n-1} \frac{u_j}{v_j}
= \frac{2^{2H}-2}{\sqrt{a^2 \; T + b^2 \; T^{2H}}}.
\end{equation*}

\bigskip
Hence, since $\displaystyle  \frac{1}{2} < H < \frac{3}{4}$ and 
$ \frac{2^{2H}-2}{\sqrt{a^2 \; T + b^2 \; T^{2H}}} > 0$, we have
$\displaystyle \lim_{n \rightarrow \infty } I_n = + \infty$. Then, we get by using (\ref{eq50}) that
\begin{equation*}
\sup_{\tau } \sum_{j=0}^{n-1} \left\| \E \Big( S^H_{t_{j+1}} -S^H_{t_j} | 
\mathcal{F}_{t_j}^{S^{H}} \Big) \right\| _1 = + \; \infty .
\end{equation*}

\bigskip
This completes the proof of the lemma.

\end{proof}

\bigskip
\begin{proposition} \label{p:28} 
The smfBm 
$\displaystyle S^{\frac{3}{4}}(a,b) $ 
is not a quasi-martingale.
\end{proposition}

\bigskip
\begin{proof} 
Since conditional expectation is a
contraction with respect to the $L^1-$~norm, we have for all $n \in \N$ and all $j=
1, ..., n-1$, 
\begin{equation*}
\left\| \E \Big( \Delta_{j+1}^n S^{\frac{3}{4}} | \Delta_{j}^n S^{\frac{3}{4}%
},..., \Delta_{1}^n S^{\frac{3}{4}} \Big) \right\| _1 \le \left\| \E \Big( %
\Delta_{j+1}^n S^{\frac{3}{4}} | \mathcal{F}_{\frac{Tj}{n}}^{S^{\frac{3}{4}%
}} \Big) \right\| _1.
\end{equation*}
Moreover, since $\displaystyle \E \Big( \Delta_{j+1}^n S^{\frac{3}{4}} |
\Delta_{j}^n S^{\frac{3}{4}},..., \Delta_{1}^n S^{\frac{3}{4}} \Big)$ is a
centered Gaussian random variable, 
\begin{equation*}
\left\| \E \Big( \Delta_{j+1}^n S^{\frac{3}{4}} | \Delta_{j}^n S^{\frac{3}{4}%
},..., \Delta_{1}^n S^{\frac{3}{4}} \Big) \right\| _1 = \sqrt{\frac{2}{\pi }}
\left\| \E \Big( \Delta_{j+1}^n S^{\frac{3}{4}} | \Delta_{j}^n S^{\frac{3}{4}%
},..., \Delta_{1}^n S^{\frac{3}{4}} \Big) \right\| _2.
\end{equation*}
Consequently, 
\begin{equation*}
\begin{array}{rcl}
\displaystyle \sum_{j=1}^{n-1} \left\| \E \Big( \Delta_{j+1}^n S^{\frac{3}{4}%
} | \mathcal{F}_{\frac{j}{n}}^{S^{\frac{3}{4}}} \Big) \right\| _1 & \ge & %
\displaystyle \sqrt{\frac{2}{\pi }} \sum_{j=1}^{n-1} \left\| \E \Big( %
\Delta_{j+1}^n S^{\frac{3}{4}} | \Delta_{j}^n S^{\frac{3}{4}},...,
\Delta_{1}^n S^{\frac{3}{4}} \Big) \right\| _2 ,%
\end{array}%
\end{equation*}

\noindent and the lemma is proved if we show that 
\begin{equation} 
 \label{eq:26}
\lim_{n \rightarrow \infty} \sum_{j=1}^{n-1} \left\| \E \Big( \Delta_{j+1}^n
S^{\frac{3}{4}} | \Delta_{j}^n S^{\frac{3}{4}},..., \Delta_{1}^n S^{\frac{3}{%
4}} \Big) \right\| _2 = + \; \infty .
\end{equation}

For $n \in \N$ and $j= 1, ..., n-1$, 
\begin{equation*}
\Big( \Delta_{j+1}^n S^{\frac{3}{4}} , \Delta_{j}^n S^{\frac{3}{4}},...,
\Delta_{1}^n S^{\frac{3}{4}} \Big)
\end{equation*}
is a Gaussian vector. Therefore,

\begin{equation}  \label{eq:27}
\E \Big( \Delta_{j+1}^n S^{\frac{3}{4}} | \Delta_{j}^n S^{\frac{3}{4}},...,
\Delta_{1}^n S^{\frac{3}{4}} \Big) = \sum_{k=1}^{j}\;  b_k \; \Delta_{k}^n \; S^{%
\frac{3}{4}} ,
\end{equation}
where the vector $\displaystyle b = \left( 
\begin{array}{c}
b_1 \\ 
\vdots \\ 
b_j%
\end{array}
\right) $ solves the system of linear equations 
\begin{equation}  \label{eq:28}
m = A b,
\end{equation}
in which $m$ is a $j-$vector whose $k-th$ component $m_k$ is 
\begin{equation*}
Cov \Big( \Delta_{j+1}^n S^{\frac{3}{4}} , \Delta_{k}^n S^{\frac{3}{4}} \Big)
\end{equation*}
and $A$ is the covariance matrix of the Gaussian vector 
\begin{equation*}
\Big( \Delta_{1}^n S^{\frac{3}{4}}, ..., \Delta_{j}^n S^{\frac{3}{4}} \Big) .
\end{equation*}

Note that $A$ is symmetric and, since the random variables 
\begin{equation*}
\Delta_{1}^n S^{\frac{3}{4}}, ..., \Delta_{j}^n S^{\frac{3}{4}}
\end{equation*}
are lineary independent, $A$ is also positive definite. It follows from (\ref{eq:27})
and (\ref{eq:28}) that

\begin{equation}  \label{eq:29}
\left\| \E \Big( \Delta_{j+1}^n S^{\frac{3}{4}} | \Delta_{j}^n S^{\frac{3}{4}%
},..., \Delta_{1}^n S^{\frac{3}{4}} \Big) \right\| _2^2 = b^T A b = m^T
A^{-1} m \ge \left\| m \right\| _2^2 \lambda ^{-1} ,
\end{equation}

\noindent where $\lambda $ is the largest eigenvalue of the matrix $A$. Set $%
I$ the identity matrix and $C = (C_{i,k})_{1 \leq i,k \leq j} $ the
covariance matrix of the increments of the sfBm with index $3/4$. We have 
\begin{equation*}
A = \frac{a^2T}{n} \; I + b^2 \; C,
\end{equation*}
\noindent and consequently

\begin{equation}  \label{eq:30}
\lambda = \frac{a^2T}{n} + b^2 \; \mu,
\end{equation}

\noindent where $\mu $ is the largest eigenvalue of the matrix $C$. We deduce also from
lemma \ref{L:13}

\begin{equation*}
\begin{array}{rcl}
\displaystyle C_{ik} & = & \displaystyle \frac{T^{3/2}}{2 \; n^{3/2}} \; \Biggl( %
\Bigl( \mid k - i \mid +1 \Bigr)^{3/2} - 2 \mid k - i \mid^{3/2} + \mid \mid
k - i \mid -1 \mid^{3/2} \\ 
\noalign{\vskip 3mm} &  & \displaystyle + 2 \; (k+i-1)^{3/2} - (k+i)^{3/2} -
(k+i-2)^{3/2} \Biggr) \\ 
\noalign{\vskip 3mm} & = & \displaystyle \frac{T^{3/2}}{2 \; n^{3/2}} \; (
E_{ik} + F_{ik} ),%
\end{array}%
\end{equation*}

where 
\begin{equation*}
E_{ik} = 2 \Biggl( \frac{ \Bigl( \mid k - i \mid +1 \Bigr)^{3/2} + \mid \mid
k - i \mid -1 \mid^{3/2} }{2} - \mid k - i \mid^{3/2} \Biggr)
\end{equation*}
and 
\begin{equation*}
F_{ik} = 2 \Biggl( (k+i-1)^{3/2} - \frac{(k+i)^{3/2} + (k+i-2)^{3/2}}{2} %
\Biggr) .
\end{equation*}
\noindent Note that the convexity of the function $\displaystyle x
\rightarrow x^{3/2}, \; x \geq 0$, implies that $\displaystyle  E_{ik} \geq
0 $ and $\displaystyle F_{ik} \leq 0$. Moreover, since $H = 3/4 > 1/2$,
corollary \ref{C:14} yields $\; C_{ik} \geq 0$.

\bigskip So, using the Gershgorin circle theorem \cite{Gol} we obtain 
\begin{equation*}
\mu \leq \max_{k=1,..,j} \; \sum_{k=1}^j \; \mid C_{ik} \mid \le \frac{T^{3/2}%
}{2 \; n^{3/2}} \max_{k=1,..,j} \; \sum_{k=1}^j E_{ik} ,
\end{equation*}
\noindent and consequently 
\begin{equation*}
\begin{array}{rcl}
\displaystyle \mu & \leq & \displaystyle \frac{T^{3/2}}{n^{3/2}} \;
\sum_{k=1}^j \; \Biggl( 2 \Biggl( \frac{ \Bigl( \mid k - 1 \mid +1 \Bigr)%
^{3/2} + \mid \mid k - 1 \mid -1 \mid^{3/2} }{2} - \mid k - 1 \mid^{3/2} %
\Biggr) \\ 
\noalign{\vskip 3mm} & = & \displaystyle \frac{T^{3/2}}{n^{3/2}} \;
\sum_{k=1}^j \; \Biggl( \Bigl( \mid k - 1 \mid +1 \Bigr)^{3/2} + \mid \mid k
- 1 \mid -1 \mid^{3/2} - 2 \mid k - 1 \mid^{3/2} \Biggr) \\ 
\noalign{\vskip 3mm} & = & \displaystyle \frac{T^{3/2}}{n^{3/2}} \;
\sum_{k^{^{\prime }}=0}^{j-1} \; \Biggl( (k^{^{\prime }} +1)^{3/2} - 2 \;
k^{^{\prime }3/2} + \mid k^{^{\prime }} -1 \mid^{3/2} \Biggr) \\ 
\noalign{\vskip 3mm} & = & \displaystyle \frac{T^{3/2}}{n^{3/2}} \; \Bigl( 1
+ j^{3/2} - (j-1)^{3/2} \Bigr) \\ 
\noalign{\vskip 3mm} & \leq & \displaystyle T^{3/2} \Bigg( \frac{1}{n^{3/2}}
+ \frac{1}{n^{3/2}} \; \max_{j-1 \leq x \leq j} \; \frac{d(x^{3/2})}{dx} %
\Bigg) \\ 
\noalign{\vskip 3mm} & \leq & \displaystyle \frac{5}{2 \; n} \; T^{3/2} .%
\end{array}%
\end{equation*}

\bigskip Hence combining equality (\ref{eq:30}) with the above result, we
obtain

\begin{equation}  \label{eq:31}
\lambda^{-1} \geq \alpha \, n,
\end{equation}
where $\displaystyle \alpha = \frac{2}{T(2a^2 + 5b^2T^{1/2})} .$\newline

Next, let us determine a suitable lower bound of $\| m \|_2^2$. From the
lemma \ref{L:13} we have

\begin{equation}  \label{eq:32}
\begin{array}{rcl}
\displaystyle \| m \|_2^2 & = & \displaystyle \sum_{k=1}^j \; \Bigl( Cov %
\Bigl( \Delta_{j+1}^n S^{\frac{3}{4}}, \Delta_k^n S^{\frac{3}{4}} \Bigr) %
\Bigr)^2 \\ 
\noalign{\vskip 3mm} & = & \displaystyle \frac{T^3b^4}{4n^3} \sum_{k=1}^j \; %
\Bigl( f_1(k) - f_2(k) \Bigr) ^2,%
\end{array}%
\end{equation}

\noindent where 
\begin{equation*}
f_1(k) = (j-k+2)^{3/2} - 2 (j-k+1)^{3/2} + (j-k)^{3/2}
\end{equation*}
\noindent and 
\begin{equation*}
f_2(k) = (j+k+1)^{3/2} - 2 (j+k)^{3/2} + (j+k-1)^{3/2}.
\end{equation*}

\bigskip

The functions $f_1$ and $f_2$ satisfy three properties, which we shall use at the end of the proof. We will state them in the following technical lemma.

\bigskip
\begin{lemma} \label{L:29}  For any $ k \in \{1,..,j \}$ \\
\begin{itemize}
\item
$ f_1(k)  \geq 0 $ and $ f_2(k) \geq 0$, 
\item
$  f_1(k)  - f_2(k) > 0$,
\item
 \begin{equation} \label{eq:33}
  f_1(k)  - f_2(k) \geq \frac{3}{4} \; \Bigl( (j-k+1)^{-1/2} - (j+k-1)^{-1/2} \Bigr) \geq 0.
  \end{equation}  
\end{itemize}
\end{lemma}

\bigskip
\begin{proof} (of lemma \ref{L:29}) The first assertion of the lemma is due to the fact that the function $\displaystyle x
\longmapsto x^{3/2}$ is convex on the interval $\displaystyle \lbrack 0, +
\infty [$.

\bigskip
Now, let us prove the second assertion of the lemma. Consider the function $g$ defined by 
\begin{equation*}
g(x) = (x+1)^{3/2} - 2 \; x^{3/2} + (x - 1)^{3/2}, \; \; x \geq 1.
\end{equation*}
Since the function $\displaystyle x \longmapsto x^{1/2}$ is concave on $%
\displaystyle \lbrack 1, + \infty [$, $g$ decreases on this interval and
consequently 
\begin{equation*}
f_1(k) = g(j-k+1) > g(j+k) = f_2(k) .
\end{equation*}

\bigskip
Finally, let us prove inequality (\ref{eq:33}). For every $a \ge 1$, let us consider the function $g_a$ defined by 
\begin{equation*}
g_a(x) = (a+x)^{3/2} - 2 \; a^{3/2} + (a - x)^{3/2}, \; \; 0 \leq x \leq 1
\leq a.
\end{equation*}

We have $\displaystyle \; g_a(0) = 0, g_a^{^{\prime }}(x) = \frac{3}{2} \; (
(a+x)^{1/2} - (a-x)^{1/2} )$ and therefore $g_a^{^{\prime }}(0) = 0$.\newline

On the otherhand, by Taylor-Lagrange theorem, we get that there exists $c
\in ]0,1[$ such that 
\begin{equation*}
g_a(1) = g_a(0) + g_a^{^{\prime }}(0) + \frac{1}{2} g_a^{"}(c) = \frac{1}{2}
g_a^{"}(c),
\end{equation*}
\noindent 
where 
$$
g_a^{"}(x) = \frac{3}{4} \; \Bigl( (a+x)^{-1/2} + (a - x)^{-1/2} \Bigr).
$$

Next, it is easy to check that the function $g_a^{"}$ increases, and
consequently 
\begin{equation*}
g_a^{"}(0) \leq g_a^{"}(c) \leq g_a^{"}(1).
\end{equation*}
So, we have 
\begin{equation*}
\frac{3}{ 4 \; a^{1/2}} \leq g_a(1) \leq \frac{3}{4 \;(a-1)^{1/2}},
\end{equation*}
and therefore
\begin{equation*}
f_1(k) = g_{j-k+1}(1) \geq \frac{3}{4} \, (j-k+1)^{-1/2} \; \; \; \mbox{and}
\; \; \; f_2(k) = g_{j+k} (1) \leq \frac{3}{4} \, (j+k-1)^{-1/2},
\end{equation*}
\noindent which ends the proof of the lemma. 

\end{proof}

\bigskip
Let us turn back to the proof of proposition \ref{p:28}. Combining (\ref%
{eq:32}) with (\ref{eq:33}), we get

\begin{equation}  \label{eq:34}
\| m \|_2^2 \geq \frac{9 \, b^4 \, T^3}{64 \, n^3} \sum_{k=2}^j \; \Bigl( %
(j-k+1)^{-1/2} - (j+k-1)^{-1/2} \Bigr)^2 .
\end{equation}

\bigskip

For every integer $j \ge 1$, let us consider the function 
\begin{equation*}
f_j(x) = (j-x+1)^{-1/2} - (j+x-1)^{-1/2}, \; 1 \leq x \leq j .
\end{equation*}

Since $f_j$ increases, we have

\begin{equation}  \label{eq:35}
\sum_{k=2}^j \; \Bigl( (j-k+1)^{-1/2} - (j+k-1)^{-1/2} \Bigr)^2 \geq
\int_1^j \; f_j(x)^2 \; dx .
\end{equation}

But

\begin{equation}  \label{eq:36}
\begin{array}{rcl}
\displaystyle \int_1^j \; f_j(x)^2 \; dx & = & \displaystyle \int_1^j \; %
\Biggl( \frac{1}{j-x+ 1} + \frac{1}{j+x-1} - 2 \frac{1}{\sqrt{j^2- (x-1)^2}} %
\Biggr) \; dx \\ 
\noalign{\vskip 3mm} & = & \displaystyle \ln (2 j -1) -2 \; \int_1^j \; 
\frac{1}{\sqrt{j^2- (x-1)^2}} \; dx \\ 
\noalign{\vskip 3mm} & = & \displaystyle \ln (2 j -1) + 2 \arccos \Bigl( 
\frac{j-1}{j} \Bigr) - \pi.%
\end{array}%
\end{equation}

\bigskip Hence, combining (\ref{eq:31}) with (\ref{eq:34}), (\ref{eq:35}) and (\ref{eq:36}), we get

\begin{equation}  \label{eq:37}
\| m \|_2^2 \; \lambda^{-1} \geq \frac{\beta }{n^2} \; \Bigl( \ln (2 j -1) +
2 \arccos \Bigl( \frac{j-1}{j} \Bigr) - \pi \Bigr) ,
\end{equation}

where $\displaystyle \beta = \frac{\alpha }{64} (9 \; T^3 \; b^4).$

\bigskip
Combining (\ref{eq:37}) with (\ref{eq:29}), we have, 
\begin{equation*}
\sum_{j=1}^{n-1} \; \| E \Bigl( \Delta_{j+1}^n S_{3/4} \mid \Delta_j^n
S_{3/4},.., \Delta_1^n S_{3/4}\Bigr) \|_2
\end{equation*}
\begin{equation*}
\geq \frac{\sqrt{\beta }}{n} \; \sum_{j=1}^{n-1} \; \sqrt{ \ln (2 j -1) + 2
\arccos \Bigl( \frac{j-1}{j} \Bigr) - \pi}.
\end{equation*}

\bigskip 
Since $\displaystyle \lim_{n \rightarrow \infty } \sqrt{ \ln (2 n
-1) + 2 \arccos \Bigl( \frac{n-1}{n} \Bigr) - \pi} = + \; \infty $, we have by C\'{e}saro
theorem
\begin{equation*}
\lim_{n \rightarrow \infty } \frac{\sqrt{\beta }}{n} \; \sum_{j=1}^{n-1} \; 
\sqrt{ \ln (2 j -1) + 2 \arccos \Bigl( \frac{j-1}{j} \Bigr) - \pi} = + \; \infty ,
\end{equation*}
which completes the proof of proposition \ref{p:28}.

\end{proof}

\newpage

\bigskip
\noindent
Charles EL-NOUTY

\bigskip
\noindent
LAGA, universit\'{e} Paris XIII, 99 avenue J-B Cl\'{e}ment, 93430 Villetaneuse,
FRANCE

\bigskip
\noindent
Email: elnouty@math.univ-paris13.fr

\vspace{2cm}
\noindent
 Mounir ZILI

\bigskip
\noindent
Preparatory Institute to the Military Academies, Research unit
UR04DN04, Avenue Mar\'echal Tito, 4029 Sousse, TUNISIA

\bigskip
\noindent
Email: zilimounir@yahoo.fr

\end{document}